\documentclass[runningheads]{llncs}
\usepackage{graphicx}
\usepackage{longtable}
\usepackage{wrapfig}
\usepackage{ucs}
\usepackage{dsfont}
\usepackage{mathrsfs}
\usepackage{mathtools}
\usepackage{amsmath}
\usepackage{amsfonts}
\usepackage{amssymb}
\usepackage{soul}
\usepackage[makeroom]{cancel}
\usepackage{bbm}
\usepackage[english]{babel}
\usepackage{ucs}
\usepackage[colorlinks=true,linkcolor=blue,citecolor=blue,urlcolor=orange]{hyperref}
\usepackage{longtable}
\usepackage{wasysym}
\usepackage{verbatim}
\usepackage{multirow}
\usepackage{bussproofs}
\usepackage{setspace}
\usepackage{graphicx}
\usepackage{stmaryrd}
\usepackage{forest}
\forestset{smullyan tableaux/.style={for tree={math content},where n children=1{!1.before computing xy={l=\baselineskip},!1.no edge}{},closed/.style={label=below:$\times$},},}
\usepackage{longtable}
\usepackage[all]{xy}
\usepackage{wrapfig}
\usepackage{xcolor}
\usepackage[colorinlistoftodos,prependcaption,textsize=small]{todonotes}
\usepackage[most]{tcolorbox}
\usepackage{enumitem}
\usepackage{tikz}
\usetikzlibrary{arrows,calc,patterns,positioning,shapes}
\usetikzlibrary{decorations.pathmorphing}
\tikzset{
modal/.style={>=stealth',shorten >=1pt,shorten <=1pt,auto,
node distance=1.5cm,semithick},
world/.style={circle,draw,minimum size=1cm},
point/.style={circle,draw,fill=black,inner sep=0.5mm},
reflexive/.style={->,in=120,out=60,loop,looseness=#1},
reflexive/.default={5},
reflexive point/.style={->,in=135,out=45,loop,looseness=#1},
reflexive point/.default={25},
}
\newcommand{\commentSabine}[1]{}

\newcommand{\Prop}{\mathtt{Prop}}

\newcommand{\BD}{\textsf{BD}}

\newcommand{\Luk}{{\mathchoice{\mbox{\rm\L}}{\mbox{\rm\L}}{\mbox{\rm\scriptsize\L}}{\mbox{\rm\tiny\L}}}}
\newcommand{\infoGsquare}{\mathsf{G}^{2\pm}_{\blacksquare,\blacklozenge}}
\newcommand{\fbinfoGsquare}{{\infoGsquare}_{\mathsf{fb}}}

\newcommand{\infobimodalLsquare}{\mathcal{L}^\neg_{\blacksquare,\blacklozenge}}
\newcommand{\TfbinfoGsquare}{\mathcal{T}\!\left(\fbinfoGsquare\right)}
\newcommand{\Str}{\mathsf{Str}}
\newcommand{\AStr}{\mathsf{AStr}}
\newcommand{\str}{\mathsf{str}}
\newcommand{\real}{\mathsf{rl}}

\newcommand{\coimplies}{\Yleft}
\newcommand{\Gsquare}{\mathsf{G}^2}
\newcommand{\KGsquare}{\mathbf{K}\mathsf{G}^2}
\newcommand{\KG}{\mathfrak{GK}}
\newcommand{\KbiG}{\mathbf{K}\mathsf{biG}}
\newcommand{\fbKGsquare}{\mathbf{K}\mathsf{G}^2_{\mathsf{fb}}}

\newcommand{\pspace}{\mathsf{PSPACE}}
\newtheorem{convention}{Convention}

\usepackage{comment}
\begin{document}

% -- spacing in align environment
% between lines
\setlength{\jot}{0pt} 
% before and after
\setlength{\abovedisplayskip}{2pt}
\setlength{\belowdisplayskip}{2pt}
\setlength{\abovedisplayshortskip}{1pt}
\setlength{\belowdisplayshortskip}{1pt}

\title{Non-standard modalities\\in paraconsistent G\"{o}del logic\thanks{The research of Marta B\'ilkov\'a was supported by the grant 22-01137S of the Czech Science Foundation. The research of Sabine Frittella and Daniil Kozhemiachenko was funded by the grant ANR JCJC 2019, project PRELAP (ANR-19-CE48-0006). This research is part of the MOSAIC project financed by the European Union's Marie Sk\l{}odowska-Curie grant No.~101007627.}}
\titlerunning{Paraconsistent non-standard modalities}
\author{Marta B\'ilkov\'a\inst{1}\orcidID{0000-0002-3490-2083} \and Sabine Frittella\inst{2}\orcidID{0000-0003-4736-8614}\and Daniil Kozhemiachenko\inst{2}\orcidID{0000-0002-1533-8034}}
\authorrunning{B\'ilkov\'a et al.}
\institute{The Czech Academy of Sciences, Institute of Computer Science, Prague\\
\email{bilkova@cs.cas.cz}
\and
INSA Centre Val de Loire, Univ.\ Orl\'{e}ans, LIFO EA 4022, France\\
\email{sabine.frittella@insa-cvl.fr, daniil.kozhemiachenko@insa-cvl.fr}}
\maketitle              % typeset the header of the contribution
\begin{abstract}
We introduce a paraconsistent expansion of the G\"{o}del logic with a De Morgan negation $\neg$ and modalities $\blacksquare$ and $\blacklozenge$. We dub the logic $\infoGsquare$ and equip it with Kripke semantics on frames with two (possibly fuzzy) relations: $R^+$ and $R^-$ (interpreted as the degree of trust in affirmations and denials by a given source) and valuations $v_1$ and $v_2$ (positive and negative support) ranging over $[0,1]$ and connected via $\neg$.

We motivate the semantics of $\blacksquare\phi$ (resp., $\blacklozenge\phi$) as infima (suprema) of both positive and negative supports of $\phi$ in $R^+$- and $R^-$-accessible states, respectively. We then prove several instructive semantical properties of $\infoGsquare$. Finally, we devise a~tableaux system for $\infoGsquare$ over finitely branching frames and establish the complexity of satisfiability and validity.
\keywords{G\"{o}del logic \and modal logic \and non-standard modalities \and constraint tableaux}
\end{abstract}
\section{Introduction\label{sec:introduction}}
When aggregating information from different sources, two of the simplest strategies are as follows: either one is sceptical and cautious regarding the information they provide thus requiring that they agree, or one is credulous and trusts their sources. In the classical setting, these two strategies can be modelled with $\Box$ and $\lozenge$ modalities defined on Kripke frames where states are sources, the accessibility relation represents references between them, and $w\vDash\phi$ is construed as ‘$w$ says that $\phi$ is true’. However, the sources can contradict themselves or be silent regarding a given question (as opposed to providing a~clear denial). Furthermore, a~source can provide a degree to their confirmation or denial. In all of these cases, classical logic struggles to formalise reasoning with such information.

\vspace{.5em}

\textbf{Paraconsistent reasoning about imperfect data}
In the situation described above, one can use the following setting. A~source $w$ gives a~statement $\phi$ two valuations over $[0,1]$: $v_1$ standing for the degree with which $w$ \emph{asserts} $\phi$ (positive support or support of truth) and $v_2$ for the degree of \emph{denial} (negative support or support of falsity). \emph{Classically}, $v_1(\phi,w)+v_2(\phi,w)=1$; if a source provides \emph{contradictory information}, then $v_1(\phi,w)+v_2(\phi,w)>1$; if the source provides \emph{insufficient information}, then $v_1(\phi,w)+v_2(\phi,w)<1$.

Now, if we account for the nonclassical information provided by the sources, the two aggregations described above can be formalised as follows. For the \emph{sceptical} case, the agent considers \emph{infima of positive and negative supports}. For the \emph{credulous aggregation}, one takes \emph{suprema of positive and negative supports}.

These two aggregation strategies were initially proposed and analysed in~\cite{BilkovaFrittellaMajerNazari2020}. There, however, they were described in a two-layered framework\footnote{We refer our readers to~\cite{BaldiCintulaNoguera2020} and~\cite{BilkovaFrittellaKozhemiachenkoMajer2023IJAR} for an exposition of two-layered modal logics.} which prohibits the nesting of modalities. Furthermore, the Belnap--Dunn logic~\cite{Belnap2019} ($\BD$) that lacks implication was chosen as the propositional fragment. In this paper, we extend that approach to the Kripke semantics to incorporate possible references between the sources and the sources' ability to give modalised statements. Furthermore, we use a paraconsistent expansion $\Gsquare$ from~\cite{BilkovaFrittellaKozhemiachenko2021} of G\"{o}del logic $\mathsf{G}$ as the propositional fragment.

\vspace{.5em}

\textbf{Formalising beliefs in modal expansions of $\mathsf{G}$}
When information is aggregated, the agent can further reason with it. For example, if one knows the degrees of certainty of two given statements, one can add them up, subtract them from one another, or compare them. In many contexts, however, an ordinary person does not represent their certainty in a given statement numerically and thus cannot conduct arithmetical operations with them. What they can do instead, is to \emph{compare} their certainty in one statement vs the other.

Thus, since G\"{o}del logic expresses order and comparisons but not arithmetic operations, it can be used as a propositional fragment of a modal logic formalising beliefs. For example, $\mathbf{K45}$ and $\mathbf{KD45}$ G\"{o}del logics can be used to formalise possibilistic reasoning since they are complete w.r.t.\ normalised and, respectively, non-normalised possibilistic frames~\cite{RodriguezTuytEstevaGodo2022}.

Furthermore, adding coimplication $\coimplies$ or, equivalently, Baaz' Delta operator $\triangle$ (cf.~\cite{Baaz1996} for details), results in bi-G\"{o}del (‘symmetric G\"{o}del’ in the terminology of~\cite{GrigoliaKiseliovaOdisharia2016}) logic that can additionally express strict order.

Modal expansions of $\mathsf{G}$ are well-studied. In particular, the Hilbert~\cite{CaicedoRodriguez2010} and Gentzen~\cite{MetcalfeOlivetti2009,MetcalfeOlivetti2011} formalisations of both $\Box$ and $\lozenge$ fragments of the modal logic $\KG$~\footnote{$\Box$ and $\lozenge$ are not interdefinable in $\KG$.} are known. There are also complete axiomatisations for both fuzzy~\cite{CaicedoRodriguez2015} and crisp~\cite{RodriguezVidal2021} bi-modal G\"{o}del logics. It is known that they and some of their extensions are both decidable and $\pspace$ complete~\cite{CaicedoMetcalfeRodriguezRogger2013,CaicedoMetcalfeRodriguezRogger2017,DieguezFernandez-Duque2023} even though they lack finite model property.

Furthermore, it is known that the addition of $\coimplies$ or $\triangle$ as well as of a paraconsistent negation $\neg$ that swaps the supports of truth and falsity does not increase the complexity. Namely, satisfiability of $\KbiG$ and $\mathsf{GTL}$ (modal and temporal bi-G\"{o}del logics, respectively) (cf.~\cite{BilkovaFrittellaKozhemiachenko2022IJCAR,BilkovaFrittellaKozhemiachenko2022IGPLarxiv} for the former and~\cite{AguileraDieguezFernandez-DuqueMcLean2022} for the latter) as well as that of $\KGsquare$ (expansion of crisp $\KG$ with $\neg$\footnote{Note that in the presence of $\neg$, $\phi\coimplies\phi'$ is definable as $\neg(\neg\phi'\rightarrow\neg\phi)$.}) are in $\pspace$.

\vspace{.5em}

\textbf{This paper}
In this paper, we consider an expansion of $\Gsquare$ with modalities $\blacksquare$ and $\blacklozenge$ that stand for the cautious and credulous aggregation strategies. We equip them with Kripke semantics, construct a sound and complete tableaux calculus, and explore their semantical and computational properties. Our inspiration comes from two sources: modal expansions of G\"{o}del logics that we discussed above and modal expansions of Belnap--Dunn logic with Kripke semantics on bi-valued frames as studied by Priest~\cite{Priest2008FromIftoIs,Priest2008}, Odintsov and Wansing~\cite{OdintsovWansing2010,OdintsovWansing2017}, and others (cf.~\cite{Drobyshevich2020} and references therein to related work in the field). In a sense, $\infoGsquare$ can be thought of as a hybrid between modal logics over $\BD$

The remaining text is organised as follows. In Section~\ref{sec:language}, we define the language and semantics of $\infoGsquare$. Then, in Section~\ref{sec:definability} we show how to define several important frame classes, in particular, finitely branching frames. We also argue for the use of $\fbinfoGsquare$ ($\infoGsquare$ over finitely branching frames) for the representation of agents' beliefs. In Section~\ref{sec:tableaux} we present a sound and complete tableaux calculus for $\fbinfoGsquare$ and in Section~\ref{sec:complexity}, we use it to show that $\fbinfoGsquare$ validity and satisfiability are $\pspace$ complete. Finally, in Section~\ref{sec:conclusion}, we wrap up the paper and provide a roadmap to future work.
\section{Logical preliminaries\label{sec:language}}
In this section, we provide semantics of $\infoGsquare$ over both fuzzy and crisp frames. To make the presentation more approachable, we begin with bi-G\"{o}del algebras.
\begin{definition}\label{def:bi-G_algebra}
The bi-G\"{o}del algebra $[0,1]_{\mathsf{G}}=\langle[0,1],0,1,\wedge_\mathsf{G},\vee_\mathsf{G},\rightarrow_{\mathsf{G}},\coimplies\rangle$ is defined as follows: for all $a,b\in[0,1]$, we have $a\wedge_\mathsf{G}b=\min(a,b)$, $a\vee_\mathsf{G}b=\max(a,b)$. The remaining operations are defined below:
\begin{align*}
a\rightarrow_\mathsf{G}b&=
\begin{cases}
1,\text{ if }a\leq b\\
b\text{ else}
\end{cases}
&
a\coimplies_\mathsf{G}b&=
\begin{cases}
0,\text{ if }a\leq b\\
a\text{ else}
\end{cases}
% \\
% {\sim_\mathsf{G}}a&=
% \begin{cases}
% 0,\text{ if }a>0\\
% 1\text{ else}
% \end{cases}
% &
% \triangle_\mathsf{G}a&=
% \begin{cases}
% 0,\text{ if }a<1\\
% 1\text{ else}
% \end{cases}
\end{align*}
\end{definition}
% It is easy to see that $\triangle_\mathsf{G}$ and ${\coimplies_\mathsf{G}}$ are interdefinable:
% \begin{align*}
% \triangle_\mathsf{G}a&=1\coimplies_\mathsf{G}(1\coimplies_\mathsf{G}a)&a\coimplies_\mathsf{G}b&=a\wedge{\sim_\mathsf{G}}\triangle_\mathsf{G}(a\rightarrow b)
% \end{align*}

We are now ready to define the language and semantics of $\infoGsquare$.
\begin{definition}\label{def:semantics}
We fix a countable set of propositional variables $\Prop$ and define the language via the following grammar.
\[\infobimodalLsquare\ni\phi\coloneqq p\in\Prop\mid\neg\phi\mid(\phi\!\wedge\!\phi)\mid(\phi\!\rightarrow\!\phi)\mid\blacksquare\phi\mid\blacklozenge\phi\]
Constants $\mathbf{0}$ and $\mathbf{1}$, disjunction $\vee$, and coimplication $\coimplies$ as well as G\"{o}del negation ${\sim}$ can be defined as expected:
\begin{align*}
\mathbf{1}&\coloneqq p\!\rightarrow\!p&\mathbf{0}&\coloneqq\neg\mathbf{1}&{\sim}\phi&\coloneqq\phi\!\rightarrow\!\mathbf{0}&\phi\!\vee\!\phi'&\coloneqq\neg(\neg\phi\!\wedge\!\neg\phi')&\phi\!\coimplies\!\phi'&\coloneqq\neg(\neg\phi'\!\!\rightarrow\!\!\neg\phi)
\end{align*}

A \emph{fuzzy bi-relational frame is a tuple} $\mathfrak{F}=\langle W,R^+,R^-\rangle$ with $W\neq\varnothing$ and $R^+,R^-:W\times W\rightarrow[0,1]$. In a \emph{crisp frame}, $R^+,R^-:W\times W\rightarrow\{0,1\}$. A~\emph{model} is a~tuple $\mathfrak{M}=\langle W,R^+,R^-,v_1,v_2\rangle$ with $\langle W,R^+,R^-\rangle$ being a frame and $v_1,v_2:\Prop\rightarrow[0,1]$ that are extended to the complex formulas as follows.
\begin{center}
\begin{tabular}{rclrcl}
$v_1(\neg\phi,w)$&$=$&$v_2(\phi,w)$&$v_2(\neg\phi,w)$&$=$&$v_1(\phi,w)$\\
$v_1(\phi\wedge\phi',w)$&$=$&$v_1(\phi,w)\wedge_\mathsf{G}v_1(\phi',w)$&$v_2(\phi\wedge\phi',w)$&$=$&$v_2(\phi,w)\vee_\mathsf{G}v_2(\phi',w)$\\
% $v_1(\phi\vee\phi',w)$&$=$&$v_1(\phi,w)\vee_\mathsf{G}v_1(\phi',w)$&$v_2(\phi\vee\phi',w)$&$=$&$v_2(\phi,w)\wedge_\mathsf{G}v_2(\phi',w)$\\
$v_1(\phi\rightarrow\phi',w)$&$=$&$v_1(\phi,w)\!\rightarrow_\mathsf{G}\!v_1(\phi',w)$&$v_2(\phi\rightarrow\phi',w)$&$=$&$v_2(\phi',w)\coimplies_\mathsf{G}v_2(\phi,w)$\\
% $v_1(\phi\coimplies\phi',w)$&$=$&$v_1(\phi,w)\coimplies_\mathsf{G}v_1(\phi',w)$&$v_2(\phi\coimplies\phi',w)$&$=$&$v_2(\phi',w)\!\rightarrow_\mathsf{G}\!v_2(\phi,w)$
\end{tabular}
\end{center}
% The semantics of modalities is as follows.
\begin{center}
\begin{tabular}{rclrcl}
$v_1(\blacksquare\phi,w)$&$=$&$\inf\limits_{w'\!\in\!W}\!\{wR^+w'\!\!\rightarrow_\mathsf{G}\!\!v_1(\phi,w')\}$
&
$v_2(\blacksquare\phi,w)$&$=$&$\inf\limits_{w'\!\in\!W}\!\{wR^-w'\!\!\rightarrow_\mathsf{G}\!\!v_2(\phi,w')\}$\\
$v_1(\blacklozenge\phi,w)$&$=$&$\sup\limits_{w'\!\in\!W}\!\{wR^+w'\!\wedge_\mathsf{G}\!v_1(\phi,w')\}$
&
$v_2(\blacklozenge\phi,w)$&$=$&$\sup\limits_{w'\!\in\!W}\!\{wR^-w'\!\wedge_\mathsf{G}\!v_2(\phi,w')\}$
\end{tabular}
\end{center}
We will further write $v(\phi,w)=(x,y)$ to designate that $v_1(\phi,w)=x$ and $v_2(\phi,w)=y$. Moreover, we set $S(w)=\{w':wSw'>0\}$.

We say that $\phi$ is \emph{$v_1$-valid on $\mathfrak{F}$} ($\mathfrak{F}\models^+\phi$) iff for every model $\mathfrak{M}$ on $\mathfrak{F}$ and every $w\in\mathfrak{M}$, it holds that $v_1(\phi,w)=1$. $\phi$ is \emph{$v_2$-valid on $\mathfrak{F}$} ($\mathfrak{F}\models^-\phi$) iff for every model $\mathfrak{M}$ on $\mathfrak{F}$ and every $w\in\mathfrak{M}$, it holds that $v_2(\phi,w)=0$. $\phi$ is \emph{strongly valid on $\mathfrak{F}$} ($\mathfrak{F}\models\phi$) iff it is $v_1$ and $v_2$-valid.

$\phi$ is $v_1$ (resp., $v_2$, strongly) \emph{$\infoGsquare$ valid} iff it is $v_1$ (resp., $v_2$, strongly) valid on every frame. We will further use $\infoGsquare$ to designate the set of all $\infobimodalLsquare$ formulas \emph{strongly valid} on every frame.
\end{definition}

Observe in the definition above that the semantical conditions governing the support of truth of $\infoGsquare$ connectives (except for $\neg$) coincide with the semantics of $\KbiG$ (cf.~\cite{BilkovaFrittellaKozhemiachenko2022IJCAR} for the detailed semantics of the latter).

\begin{example}\label{example:restaurant}
A tourist ($t$) wants to go to a restaurant and asks their two friends ($f_1$ and $f_2$) to describe their impressions regarding the politeness of the staff ($s$) and the quality of the desserts ($d$). Of course, the friends' opinions are not always internally consistent, nor is it always the case that one or the other even noticed whether the staff was polite or was eating desserts. Furthermore, $t$ trusts their friends to different degrees when it comes to their positive and negative opinions. The situation is depicted in Fig.~\ref{fig:restaurant}.

The first friend says that half of the staff was really nice but the other half is unwelcoming and rude and that the desserts (except for the tiramisu and souffl\'{e}) are tasty. The second friend, unfortunately, did not have the desserts at all. Furthermore, even though, they praised the staff, they also said that the manager was quite obnoxious.

The tourist now makes up their mind. If they are sceptical w.r.t.\ $s$ and $d$, they look for \emph{trusted rejections}\footnote{We differentiate between a~\emph{rejection} which we treat as \emph{lack of support} and a~\emph{denial, disproof, refutation, counterexample}, etc.\ which we interpret as the \emph{negative support}.} of both positive and negative supports of $s$ and $d$. Thus $t$ uses the values of $R^+$ and $R^-$ as thresholds above which the information provided by the source does not count as a trusted enough rejection. In our case, we have $v(\blacksquare s,t)=(0.5,0.5)$ and $v(\blacksquare d,t)=(0,0)$. On the other hand, if $t$ is credulous, they look for \emph{trusted} confirmations \emph{of both positive and negative supports} and use $R^+$ and $R^-$ as thresholds up to which they accept the information provided by the source. Thus, we have $v(\blacklozenge s,t)=(0.7,0.4)$ and $v(\blacklozenge d,t)=(0.7,0.3)$.
\end{example}

\begin{figure}
\[\xymatrix{f_1:\txt{$s=(0.5,0.5)$\\$d=(0.7,0.3)$}~&&~t~\ar[rr]^(.3){(0.7,0.2)}\ar[ll]_(.3){(0.8,0.9)}&&~f_2:\txt{$s=(1,0.4)$\\$d=(0,0)$}}\]
\caption{$(x,y)$ stands for $wR^+w'=x,wR^-w'=y$. $R^+$ (resp., $R^-$) is interpreted as the tourist's threshold of trust in positive (negative) statements by the friends.}
\label{fig:restaurant}
\end{figure}
\noindent
\begin{minipage}{0.75\linewidth}
~\quad More formally, note that we can combine $v_1$ and $v_2$ into a~single valuation (denoted with $\bullet$) on the following bi-lattice on the right. Now, if we let $\sqcap$ and $\sqcup$ be the meet and join w.r.t.\ the rightward order, it is clear that $\blacksquare$ can be interpreted as an infinitary $\sqcap$ and $\blacklozenge$ as an infinitary $\sqcup$ across the accessible states, respectively.
\end{minipage}
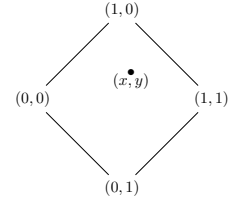
\begin{minipage}{0.25\linewidth}
\resizebox{1\linewidth}{!}{
\begin{tikzpicture}[>=stealth,relative]
\node (U1) at (0,-2) {$(0,1)$};
\node (U2) at (-2,0) {$(0,0)$};
\node (U3) at (2,0) {$(1,1)$};
\node (U4) at (0,2) {$(1,0) $};
\node (U5) at (0.2,0.6) {$\bullet$};
\node (U6) at (0.2,0.4) {$(x,y)$};
\path[-,draw] (U1) to (U2);
\path[-,draw] (U1) to (U3);
\path[-,draw] (U2) to (U4);
\path[-,draw] (U3) to (U4);
\end{tikzpicture}}
\end{minipage}

From here, it is expected that $\blacksquare$ and $\blacklozenge$ are not normal in the following sense: $\blacksquare(p\wedge q)\leftrightarrow(\blacksquare p\wedge\blacksquare q)$, $\blacksquare\mathbf{1}$, $\blacklozenge(p\vee q)\leftrightarrow(\blacklozenge p\vee\blacklozenge q)$, and $\blacklozenge\mathbf{0}\leftrightarrow\mathbf{0}$ are not valid.

Finally, we have called $\infoGsquare$ ‘paraconsistent’. In this paper, we consider the logic to be a set of valid formulas. It is clear that the explosion principle for $\rightarrow$ --- $(p\wedge\neg p)\rightarrow q$ --- is not valid. Furthermore, in contrast to $\mathbf{K}$, it is possible to believe in a~contradiction without \emph{believing in every statement}: $\blacklozenge(p\wedge\neg p)\rightarrow\blacklozenge q$ and $\blacksquare(p\wedge\neg p)\rightarrow\blacksquare q$ are not valid.

We end the section by proving that $\blacklozenge$ and $\blacksquare$ \emph{are not interdefinable}.
\begin{theorem}\label{theorem:nondefinability}
$\blacksquare$ and $\blacklozenge$ are not interdefinable.
\end{theorem}
\begin{proof}
Denote with $\mathcal{L}_\blacksquare$ and $\mathcal{L}_\blacklozenge$ the $\blacklozenge$- and $\blacksquare$-free fragments of $\infobimodalLsquare$. We build a pointed model $\langle\mathfrak{M},w\rangle$ s.t.\ there is no $\blacklozenge$-free formula that has the same value at $w$ as $\blacksquare p$ (and vice versa). Consider Fig.~\ref{fig:nondefinability}.

\begin{figure}
\[\xymatrix{w_1:p=\left(\frac{2}{3},\!\frac{1}{2}\right)&\ar[l]\ar[r]w_0:p=(1,\!0)&w_2:p=\left(\frac{1}{3},\!\frac{1}{4}\right)}\]
\caption{All variables have the same values in all states exemplified by $p$. $R^+=R^-$, $v(\blacksquare p,w_0)=\left(\frac{1}{3},\frac{1}{4}\right)$, $v(\blacklozenge p,w_0)=\left(\frac{2}{3},\frac{1}{2}\right)$.}
\label{fig:nondefinability}
\end{figure}

One can check by induction that if $\phi\in\infobimodalLsquare$, then
\begin{align*}
v(\phi,w_1)&\in\left\{(0;1),\left(\frac{1}{2};\frac{2}{3}\right),\left(\frac{2}{3};\frac{1}{2}\right),(0;0),(1;1),(1;0)\right\}
\end{align*}
\begin{align*}
v(\phi,w_2)&\in\left\{(0;1),\left(\frac{1}{4};\frac{1}{3}\right),\left(\frac{1}{3};\frac{1}{4}\right),(0;0),(1;1),(1;0)\right\}
\end{align*}
Moreover, on the single-point irreflexive frame whose only state is $u$, it holds for every $\phi(p)\in\infobimodalLsquare$, $v(\phi,u)\in\{v(p,u),v(\neg p,u),(1,0),(1,1),(0,0),(0,1)\}$.

Thus, for every $\blacklozenge$-free $\chi$ and every $\blacksquare$-free $\psi$ it holds that
\begin{align*}
v(\blacksquare\chi,w_0)&\in\left\{(0;1),\left(\frac{1}{3};\frac{1}{4}\right),\left(\frac{1}{4};\frac{1}{3}\right),(0;0),(1;1),(1;0)\right\}=X
\end{align*}
\begin{align*}
v(\blacklozenge\psi,w_0)&\in\left\{(0;1),\left(\frac{1}{2};\frac{2}{3}\right),\left(\frac{2}{3};\frac{1}{2}\right),(0;0),(1;1),(1;0)\right\}=Y
\end{align*}

Since $X$ and $Y$ are closed w.r.t.\ propositional operations, it is now easy to check by induction that for every $\chi'\in\mathcal{L}_\blacksquare$ and $\psi'\in\mathcal{L}_\blacklozenge$, $v(\chi',w_0)\in X$ and $v(\psi',w_0)\in Y$.
\end{proof}
\section{Frame definability\label{sec:definability}}
In this section, we explore some classes of frames that can be defined in $\infobimodalLsquare$. However, since $\blacksquare$ and $\blacklozenge$ are non-normal and since we have two independent relations on frames, we expand the traditional notion of modal definability.
\begin{definition}\label{def:+-definability}
\begin{enumerate}[noitemsep,topsep=2pt]
\item[]
\item $\phi$ \emph{positively defines} a class of frames $\mathbb{F}$ iff for every $\mathfrak{F}$, it holds that \emph{$\mathfrak{F}\models^+\phi$ iff $\mathfrak{F}\in\mathbb{F}$}.
\item $\phi$ \emph{negatively defines} a class of frames $\mathbb{F}$ iff for every $\mathfrak{F}$, every $w\in\mathfrak{F}$, it holds that \emph{$\mathfrak{F}\models^-\phi$ iff $\mathfrak{F}\in\mathbb{F}$}.
\item $\phi$ \emph{(strongly) defines} a class of frames $\mathbb{F}$ iff for every $\mathfrak{F}$, it holds that $\mathfrak{F}\in\mathbb{F}$ iff $\mathfrak{F}\models\phi$.
\end{enumerate}
\end{definition}

With the help of the above definition, we can show that every class of frames definable in $\KbiG$ is \emph{positively definable} in $\infoGsquare$.
\begin{definition}\label{def:+-framecounterparts}
Let $\mathfrak{F}=\langle W,S\rangle$ be a (fuzzy or crisp) frame.
\begin{enumerate}[noitemsep,topsep=2pt]
\item An \emph{$R^+$-counterpart of $\mathfrak{F}$} is any bi-relational frame $\mathfrak{F}^+=\langle W,S,R^-\rangle$.
\item An \emph{$R^-$-counterpart of $\mathfrak{F}$} is any bi-relational frame $\mathfrak{F}^+=\langle W,R^+,S\rangle$.
\end{enumerate}
\end{definition}
\begin{convention}\label{conv:blackcounterpart}
Let $\phi$ be over $\{\wedge,\vee,\rightarrow,\coimplies,\Box,\lozenge\}$.
\begin{enumerate}[noitemsep,topsep=2pt]
\item We denote with $\phi^{+\bullet}$ the formula obtained from $\phi$ by replacing all $\Box$'s and $\lozenge$'s with $\blacksquare$'s and $\blacklozenge$'s.
\item We denote with $\phi^{-\bullet}$ the formula obtained from $\phi$ by replacing all $\Box$'s and $\lozenge$'s with $\neg\blacksquare\neg$'s and $\neg\blacklozenge\neg$'s.
\end{enumerate}
\end{convention}
\begin{theorem}\label{theorem:blackcounterparts}
Let $\mathfrak{F}=\langle W,S\rangle$ and let $\mathfrak{F}^+$ and $\mathfrak{F}^-$ be its $R^+$ and $R^-$ counterparts. Then, for any $\phi$ be over $\{\wedge,\vee,\rightarrow,\coimplies,\Box,\lozenge\}$, it holds that
\[\mathfrak{F}\models_{\KbiG}\phi\quad\text{iff}\quad\mathfrak{F}^+\models^+\phi^{+\bullet}\quad\text{iff}\quad\mathfrak{F}^-\models^+\phi^{-\bullet}\]
\end{theorem}
\begin{proof}
Since the semantics of $\KbiG$ connectives is identical to $v_1$ conditions of Definition~\ref{def:semantics}, we only prove that $\mathfrak{F}\models\phi$ iff $\mathfrak{F}^-\models^+\phi^{-\bullet}$. It suffices to prove by induction the following statement.
\begin{center}
\emph{Let $\mathbf{v}$ be a $\KbiG$ valuation on $\mathfrak{F}$, $\mathbf{v}(p,w)=v_1(p,w)$ for every $w\in\mathfrak{F}$, and $v_2$ be arbitrary. Then $\mathbf{v}(\phi,w)=v_1(\phi^{-\bullet},w)$ for every $\phi$.}
\end{center}
The case of $\phi=p$ holds by Convention~\ref{conv:blackcounterpart}, the cases of propositional connectives are straightforward. Consider $\phi=\Box\chi$. We have that $\phi^{-\bullet}=\neg\blacksquare\neg(\chi^{-\bullet})$ and thus
\begin{align*}
v_1(\neg\blacksquare\neg(\chi^{-\bullet}),w)&=v_2(\blacksquare\neg(\chi^{-\bullet}),w)\\
&=\inf\limits_{w'\in W}\{wSw'\rightarrow_\mathsf{G}v_2(\neg(\chi^{-\bullet}))\}\\
&=\inf\limits_{w'\in W}\{wSw'\rightarrow_\mathsf{G}v_1(\chi^{-\bullet})\}\\
&=\inf\limits_{w'\in W}\{wSw'\rightarrow_\mathsf{G}\mathbf{v}(\chi)\}\tag{by IH}\\
&=\mathbf{v}(\Box\chi,w)
\end{align*}
\end{proof}

The above theorem allows us to \emph{positively} define in $\infoGsquare$ all classes of frames that are definable in $\KbiG$. In particular, all $\mathbf{K}$-definable frames are positively definable. Moreover, it follows that $\infoGsquare$ (as $\KG$ and $\KbiG$) lacks the finite model property: ${\sim}\Box(p\vee{\sim}p)$ is false on every finite frame, and thus, ${\sim}\blacksquare(p\vee{\sim}p)$ is too. On the other hand, there are infinite models satisfying this formula as shown below ($R^+$ and $R^-$ are crisp).
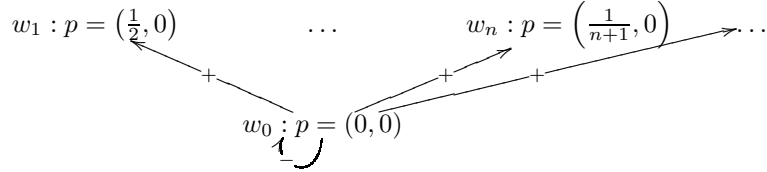
\[\xymatrix{w_1:p=\left(\frac{1}{2},0\right)&\ldots&w_n:p=\left(\frac{1}{n+1},0\right)&\ldots\\&w_0:p=(0,0)\ar[ul]|{+}\ar[ur]|{+}\ar[urr]|{+}\ar@(d,l)|{-}&&}\]

Furthermore, Theorem~\ref{theorem:blackcounterparts} gives us a degree of flexibility. For example, one can check that $\neg\blacksquare\neg(p\vee q)\rightarrow(\neg\blacksquare\neg p\vee\neg\blacklozenge\neg q)$ positively defines frames with crisp $R^-$ but not necessarily crisp $R^+$. This models a situation when an agent \emph{completely (dis)believes} in denials given by their sources while may have some degree of trust between $0$ and $1$ when the sources assert something. Let us return to Example~\ref{example:restaurant}.
\begin{example}\label{example:illnesscontinued}
Assume that the tourist \emph{completely trusts} the negative (but not positive) opinions of their friends. Thus, instead of Fig.~\ref{fig:restaurant}, we have the following model.
\[\xymatrix{f_1:\txt{$s=(0.5,0.5)$\\$d=(0.7,0.3)$}~&&~t~\ar[rr]^(.3){(0.7,1)}\ar[ll]_(.3){(0.8,1)}&&~f_2:\txt{$s=(1,0.4)$\\$d=(0,0)$}}\]

The new values for the cautious and credulous aggregation are as follows: $v(\blacksquare s,t)=(0.5,0.4)$, $v(\blacksquare d,t)=(0,0)$, $v(\blacklozenge s,t)=(0.7,0.5)$, and $v(\blacklozenge d,t)=(0.7,0.3)$.
\end{example}

Furthermore, the agent can trust the sources to the same degree no matter whether they confirm or deny statements. This can be modelled with \emph{mono-relational} frames where $R^+\!=\!R^-$. We show that they are \emph{strongly definable}.
\begin{theorem}\label{theorem:1relational}
$\mathfrak{F}$ is mono-relational iff $\mathfrak{F}\models\blacksquare\neg p\leftrightarrow\neg\blacksquare p$ and $\mathfrak{F}\models\blacklozenge\neg p\leftrightarrow\neg\blacklozenge p$.
\end{theorem}
\begin{proof}
Let $\mathfrak{F}$ be mono-relational and $R^+=R^-=R$. Now observe that
\begin{align*}
v_i(\blacksquare\neg p,w)&=\inf\limits_{w'\in W}\{wRw'\rightarrow_\mathsf{G}v_i(\neg p,w')\}\tag{$i\in\{1,2\}$}\\
&=\inf\limits_{w'\in W}\{wRw'\rightarrow_\mathsf{G}v_j(p,w')\}\tag{$i\neq j$}\\
&=v_j(\blacksquare p,w)\\
&=v_i(\neg\blacksquare p,w)
\end{align*}
For the converse, let $R^+\!\neq\!R^-$ and, in particular, $wR^+w'\!=\!x$ and $wR^-w'\!=\!y$. Assume w.l.o.g.\ that $x>y$. We set the valuation of $p$: $v(p,w')=(x,y)$ and for every $w''\neq w'$, we have $v(p,w'')=(1,1)$. It is clear that $v(\neg\blacksquare p,w)=(1,1)$. On the other hand, $v(\neg p,w')=(y,x)$, whence $v_1(\blacksquare\neg p)\neq1$.

The case of $\blacklozenge$ can be tackled in a dual manner.
\end{proof}

In the remainder of the paper, we will be concerned with $\fbinfoGsquare$ --- $\infoGsquare$ over finitely branching (both fuzzy and crisp) frames. This is for several reasons. First, in the context of formalising beliefs and reasoning with data acquired from sources, it is reasonable to assume that every source refers to only a finite number of other sources and that agents have access to a finite number of sources as well. This assumption is implicit in many classical epistemic and doxastic logics since they are often complete w.r.t.\ finitely branching models~\cite{FaginHalpernMosesVardi2003}, although cannot \emph{define} them. Second, in the finitely branching models, the values of modal formulas are \emph{witnessed}: if $v_i(\blacksquare\phi,w)=x<1$, then, $v_i(\phi,w')=x$ for some $w'$, and if $v_i(\blacklozenge\phi,w)=x$, then $wRw'=x$ or $v_i(\phi,w')=x$ for some $w'$. Intuitively, this means that the degree of $w$'s certainty in $\phi$ is purely based on the information acquired from sources and from its degree of trust in those. Finally, the restriction to finitely branching frames allows for the construction of a simple constraint tableaux calculus that can be used in establishing the complexity valuation.
\section{Tableaux calculus\label{sec:tableaux}}
In this section, we construct a sound and complete constraint tableaux system $\mathcal{T}\left(\fbinfoGsquare\right)$ for $\fbinfoGsquare$. The first constraint tableaux were proposed in~\cite{Haehnle1992,Haehnle1994,Haehnle1999} as a decision procedure for the \L{}ukasiewicz logic $\Luk$. A similar approach for the Rational Pawe\l{}ka logic was proposed in~\cite{diLascioGisolfi2005}. In~\cite{BilkovaFrittellaKozhemiachenko2021}, we constructed constraint tableaux for $\Luk^2$ and $\Gsquare$ --- the paraconsistent expansions of $\Luk$ and $\mathsf{G}$, and in~\cite{BilkovaFrittellaKozhemiachenko2022IJCAR} for modal expansions of the bi-G\"{o}del logic and $\Gsquare$.

Constraint tableaux are \emph{analytic} in the sense that their rules have subformula property. Moreover, they provide an easy way of the countermodel extraction from complete open branches. Furthermore, while the propositional connectives of $\Gsquare$ allow for the construction of an analytic proof system, e.g., a~display calculus extending that of $\mathsf{I}_4\mathsf{C}_4$\footnote{This logic was introduced several times: in~\cite{Wansing2008}, then in~\cite{Leitgeb2019}, and further studied in~\cite{OdintsovWansing2021}. It is, in fact, the propositional fragment of Moisil's modal logic~\cite{Moisil1942}. We are grateful to Heinrich Wansing who pointed this out to us.}~\cite{Wansing2008}, the modal ones are not dual to one another w.r.t.\ $\neg$ nor the G\"{o}del negation ${\sim}$. Thus, it is unlikely that an elegant (hyper-)sequent or display calculus for $\infoGsquare$ or $\fbinfoGsquare$ can be constructed.

The next definitions are adapted from~\cite{BilkovaFrittellaKozhemiachenko2022IJCAR}.
\begin{definition}\label{def:TfbinfoGsquare}
We fix a set of state-labels $\mathsf{W}$ and let $\lesssim\in\!\{<,\leqslant\}$ and $\gtrsim\in\!\{>,\geqslant\}$. Let further $w\!\in\!\mathsf{W}$, $\mathbf{x}\!\in\!\{1,2\}$, $\phi\!\in\!\infobimodalLsquare$, and $c\!\in\!\{0,1\}$. A~\emph{structure} is either $w\!:\!\mathbf{x}\!:\!\phi$, $c$, $w\mathsf{R}^+w'$, or $w\mathsf{R}^+w'$. We denote the set of structures with $\Str$. Structures of the form $w\!:\!\mathbf{x}\!:\!p$, $w\mathsf{R}^+w'$, and $w\mathsf{R}^-w'$ are called \emph{atomic} (denoted $\AStr$).

We define a \emph{constraint tableau} as a downward branching tree whose branches are sets containing constraints $\mathfrak{X}\lesssim\mathfrak{X'}$ ($\mathfrak{X},\mathfrak{X'}\in\Str$). Each branch can be extended by an application of a~rule\footnote{If $\mathfrak{X}\!<\!1,\mathfrak{X}\!<\!\mathfrak{X}'\!\in\!\mathcal{B}$ or $0\!<\!\mathfrak{X}',\mathfrak{X}\!<\!\mathfrak{X}'\!\in\!\mathcal{B}$, the rules are applied only to $\mathfrak{X}\!<\!\mathfrak{X}'$.} below (bars denote branching, $i,j\in\{1,2\}$, $i\neq j$).
\[\scriptsize{\begin{array}{cccc}
\neg_i\!\lesssim\!\dfrac{w\!:\!i\!:\!\neg\phi\!\lesssim\!\mathfrak{X}}{w\!:\!j\!:\!\phi\!\lesssim\!\mathfrak{X}}
&
\neg_i\!\gtrsim\!\dfrac{w\!:\!i\!:\!\neg\phi\!\gtrsim\!\mathfrak{X}}{w\!:\!j\!:\!\phi\!\gtrsim\!\mathfrak{X}}
&
\rightarrow_1\!\leqslant\!\dfrac{w\!:\!1\!:\!\phi\!\rightarrow\!\phi'\!\leqslant\!\mathfrak{X}}{\mathfrak{X}\!\geqslant\!{1}\left|\begin{matrix}\mathfrak{X}\!<\!{1}\\w\!:\!1\!:\!\phi'\!\leqslant\!\mathfrak{X}\\w\!:\!1\!:\!\phi\!>\!w\!:\!1\!:\!\phi'\end{matrix}\right.}
&
\rightarrow_2\!\geqslant\!\dfrac{w\!:\!2\!:\!\phi\rightarrow\phi'\!\geqslant\!\mathfrak{X}}{\mathfrak{X}\!\leqslant\!{0}\left|\begin{matrix}\mathfrak{X}\!>\!{0}\\w\!:\!2\!:\!\phi'\!\geqslant\!\mathfrak{X}\\w\!:\!2\!:\!\phi'\!>\!w\!:\!2\!:\!\phi\end{matrix}\right.}
\end{array}}\]

\[\scriptsize{\begin{array}{cccc}
\wedge_1\!\gtrsim\!\dfrac{w\!:\!1\!:\!\phi\!\wedge\!\phi'\!\gtrsim\!\mathfrak{X}}{\begin{matrix}w\!:\!1\!:\!\phi\!\gtrsim\!\mathfrak{X}\\w\!:\!1\!:\!\phi'\!\gtrsim\!\mathfrak{X}\end{matrix}}
&
\wedge_2\!\lesssim\!\dfrac{w\!:\!2\!:\!\phi\!\wedge\!\phi'\!\lesssim\!\mathfrak{X}}{\begin{matrix}w\!:\!2\!:\!\phi\!\lesssim\!\mathfrak{X}\\w\!:\!2\!:\!\phi'\!\lesssim\!\mathfrak{X}\end{matrix}}
&
\rightarrow_1\!<\!\dfrac{w\!:\!1\!:\!\phi\rightarrow\phi'\!<\!\mathfrak{X}}{\begin{matrix}w\!:\!1\!:\!\phi'\!<\!\mathfrak{X}\\w\!:\!1\!:\!\phi\!>\!w\!:\!1\!:\!\phi'\end{matrix}}
&
\rightarrow_2\!>\!\dfrac{w\!:\!2\!:\!\phi\rightarrow\phi'\!>\!\mathfrak{X}}{\begin{matrix}w\!:\!2\!:\!\phi'\!>\!\mathfrak{X}\\w\!:\!2\!:\!\phi'\!>\!w\!:\!2\!:\!\phi\end{matrix}}
\end{array}}\]

\[\scriptsize{\begin{array}{cc}
\wedge_1\!\lesssim\!\dfrac{w\!:\!1\!:\!\phi\wedge\phi'\!\lesssim\!\mathfrak{X}}{w\!:\!1\!:\!\phi\!\lesssim\!\mathfrak{X}\mid w\!:\!1\!:\!\phi'\!\lesssim\!\mathfrak{X}}
&\quad
\wedge_2\!\gtrsim\!\dfrac{w\!:\!2\!:\!\phi\wedge\phi'\!\gtrsim\!\mathfrak{X}}{w\!:\!2\!:\!\phi\!\gtrsim\!\mathfrak{X}\mid w\!:\!2\!:\!\phi'\!\gtrsim\!\mathfrak{X}}
\end{array}}\]

\[\scriptsize{\begin{array}{cc}
\rightarrow_1\!\gtrsim\!\dfrac{w\!:\!1\!:\!\phi\!\rightarrow\!\phi'\!\gtrsim\!\mathfrak{X}}{w\!:\!1\!:\!\phi\!\leqslant\!w\!:\!1\!:\!\phi'\mid w\!:\!1\!:\!\phi'\!\gtrsim\!\mathfrak{X}}&\rightarrow_2\!\lesssim\!\dfrac{w\!:\!2\!:\!\phi\rightarrow\phi'\!\lesssim\!\mathfrak{X}}{w\!:\!2\!:\!\phi'\!\leqslant\!w\!:\!2\!:\!\phi\mid w\!:\!2\!:\!\phi'\!\lesssim\!\mathfrak{X}}
\end{array}}\]

\[\scriptsize{\begin{array}{ccc}
\blacksquare_i\!\!\gtrsim\!\dfrac{w\!:\!i\!:\!\blacksquare\phi\!\gtrsim\!\mathfrak{X}}{w'\!:\!i\!:\!\phi\gtrsim\mathfrak{X}\mid w\mathsf{S}w'\!\leqslant\!w'\!:\!i\!:\!\phi}
&\quad
\blacksquare_i\!\!\leqslant\!\dfrac{w\!:\!i\!:\!\blacksquare\phi\!\leqslant\!\mathfrak{X}}{\mathfrak{X}\geqslant1\left|\begin{matrix}\mathfrak{X}\!<\!1\\w\mathsf{S}w''\!>\!w''\!:\!i\!:\!\phi\\w''\!:\!:\!i\!:\!\phi\leqslant\mathfrak{X}\end{matrix}\right.}
&\quad
\blacksquare_i\!\!<\!\dfrac{w\!:\!i\!:\!\blacksquare\phi\!<\!\mathfrak{X}}{\begin{matrix}w\mathsf{S}w''\!>\!w''\!:\!i\!:\!\phi\\w''\!:\!:\!i\!:\!\phi\!<\!\mathfrak{X}\end{matrix}}
\end{array}}\]

\[\scriptsize{\begin{array}{ccc}
\blacklozenge_i\!\!\gtrsim\!\dfrac{w\!:\!i\!:\!\blacklozenge\phi\!\gtrsim\!\mathfrak{X}}{\begin{matrix}w\mathsf{S}w''\!\gtrsim\!\mathfrak{X}\\w''\!:\!i\!:\!\phi\!\gtrsim\!\mathfrak{X}\end{matrix}}
&\quad
\blacklozenge_i\!\!\lesssim\!\dfrac{w\!:\!i\!:\!\blacklozenge\phi\!\lesssim\!\mathfrak{X}}{w'\!:\!i\!:\!\phi\lesssim\mathfrak{X}\mid w\mathsf{S}w'\!\lesssim\!\mathfrak{X}}&\quad
\left[\begin{matrix}
w''\text{ is fresh on the branch}\\
\text{if }i\!=\!1,\text{ then }\mathsf{S}\!=\!\mathsf{R}^+\\\text{if }i\!=\!2,\text{ then }\mathsf{S}\!=\!\mathsf{R}^-\\\text{in }\blacksquare_i\!\gtrsim,\blacklozenge_i\!\lesssim~w\mathsf{S}w'\text{ occurs on the branch}
\end{matrix}\right]
\end{array}}\]

A tableau's branch $\mathcal{B}$ is \emph{closed} iff one of the following conditions applies:
\begin{itemize}[noitemsep,topsep=2pt]
\item the transitive closure of $\mathcal{B}$ under $\lesssim$ contains $\mathfrak{X}<\mathfrak{X}$;
\item ${0}\geqslant{1}\in\mathcal{B}$, or $\mathfrak{X}>{1}\in\mathcal{B}$, or $\mathfrak{X}<{0}\in\mathcal{B}$.
\end{itemize}
A tableau is \emph{closed} iff all its branches are closed. We say that there is a \emph{tableau proof} of $\phi$ iff there are closed tableaux starting from $w\!:\!1\!:\!\phi<1$ and $w\!:\!2\!:\!\phi>0$.

An open branch $\mathcal{B}$ is \emph{complete} iff the following condition is met.
\begin{itemize}[noitemsep,topsep=2pt]
\item[$*$]\emph{If all premises of a rule occur on $\mathcal{B}$, then its one conclusion\footnote{Note that branching rules have \emph{two} conclusions.} occurs on~$\mathcal{B}$.}
\end{itemize}
\end{definition}

\begin{convention}\label{conv:TG2meaning}
The table below summarises the interpretations of entries.
\begin{center}
\begin{tabular}{c|c}
\textbf{entry}&\textbf{interpretation}\\\hline
$w\!:1\!:\!\phi\leqslant w'\!:2\!:\!\phi'$&$v_1(\phi,w)\leq v_2(\phi',w')$\\
$w\!:\!2\!:\!\phi\leqslant c$&$v_2(\phi,w)\leq c$ with $c\in\{0,1\}$\\
$w\mathsf{R}^-w'\leqslant w'\!:2\!:\!\phi$&$wR^-w'\leq v_2(\phi,w')$
\end{tabular}
\end{center}
\end{convention}

\begin{definition}[Branch realisation]\label{G2branchsatisfaction}
A model $\mathfrak{M}=\langle W,R^+,R^-,v_1,v_2\rangle$ with $W=\{w:w\text{ occurs on }\mathcal{B}\}$ \emph{realises a~branch $\mathcal{B}$} of a tableau iff
% the following conditions hold ($\mathbf{x},\mathbf{x}'\in\{1,2\}$, ${c}\in\{0,1\}$, $\langle w,w'\rangle\in W\times W$, $\mathsf{S},\mathsf{S}'\in\{\mathsf{R}^+,\mathsf{R}^+\}$).
% \begin{itemize}[noitemsep,topsep=2pt]
% \item $v_\mathbf{x}(\phi,w)\leq v_{\mathbf{x}'}(\phi',w')$ for any $w:\mathbf{x}:\phi\leqslant w':\mathbf{x}':\phi'\in\mathcal{B}$.
% \item $v_\mathbf{x}(\phi,w)\leq c$ for any $w:\mathbf{x}:\phi\leqslant{c}\in\mathcal{B}$.
% \item If $w\mathsf{R}^+w'\notin\mathcal{B}$ ($w\mathsf{R}^-w'\notin\mathcal{B}$), then $wR^+w'=0$.
% \item If $w\mathsf{S}w'\leqslant w:\mathbf{x}:\phi$ (resp., $w\mathsf{S}w'\leqslant c$, and $w\mathsf{S}w'\leqslant w''\mathsf{S}'w'''$), then $w\mathsf{S}w'\leqslant w\!:\!\mathbf{x}\!:\!\phi$ (resp., $w\mathsf{S}w'\leqslant c$, and $w\mathsf{S}w'\leqslant w''\mathsf{S}'w'''$).
% \end{itemize}
there is a function $\real:\Str\rightarrow[0,1]$ s.t.\ for every $\mathfrak{X},\mathfrak{Y},\mathfrak{Y}',\mathfrak{Z},\mathfrak{Z}'\in\Str$ with $\mathfrak{X}=w:\mathbf{x}:\phi$, $\mathfrak{Y}=w_i\mathsf{R}^+w_j$, and $\mathfrak{Y}'=w'_i\mathsf{R}^-w'_j$ the following holds ($\mathbf{x}\in\{1,2\}$, ${c}\in\{0,1\}$).
\begin{itemize}[noitemsep,topsep=2pt]
\item If $\mathfrak{Z}\lesssim\mathfrak{Z}'\in\mathcal{B}$, then $\real(\mathfrak{Z})\lesssim\real(\mathfrak{Z}')$.
\item $\real(\mathfrak{X})=v_\mathbf{x}(\phi,w)$, $\real(c)=c$, $\real(\mathfrak{Y})=w_iR^+w_j$, $\real(\mathfrak{Y}')=w'_iR^-w'_j$
\end{itemize}
\end{definition}

To facilitate the understanding of the rules, we give an example of a~failed tableau proof and extract a~coun\-ter-mo\-del. The proof goes as follows: first, we apply all the possible propositional rules, then the modal rules that introduce new states, and then those that use the states already on the branch. We repeat the process until all structures are decomposed into atomic ones.
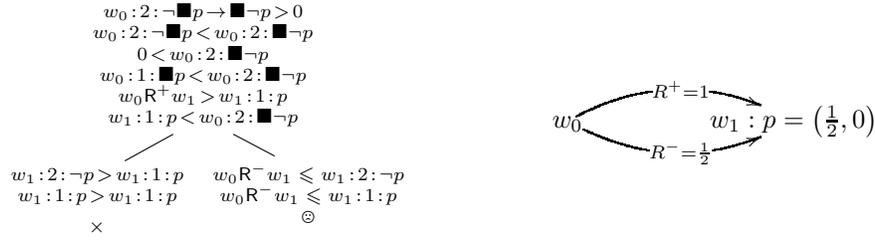
\begin{center}
\begin{minipage}{.45\linewidth}
\scriptsize{
\begin{forest}
smullyan tableaux
[w_0\!:\!2\!:\!\neg\blacksquare p\!\rightarrow\!\blacksquare\neg p\!>\!0
[w_0\!:\!2\!:\!\neg\blacksquare p\!<\!w_0\!:\!2\!:\!\blacksquare\neg p
[0\!<\!w_0\!:\!2\!:\!\blacksquare\neg p
[w_0\!:\!1\!:\!\blacksquare p\!<\!w_0\!:\!2\!:\!\blacksquare\neg p
[w_0\mathsf{R}^+w_1\!>\!w_1\!:\!1\!:\!p
[w_1\!:\!1\!:\!p\!<\!w_0\!:\!2\!:\!\blacksquare\neg p
[w_1\!:\!2\!:\!\neg p\!>\!w_1\!:\!1\!:\!p[w_1\!:\!1\!:\!p\!>\!w_1\!:\!1\!:\!p,closed]]
[w_0\mathsf{R}^-w_1\leqslant w_1\!:\!2\!:\!\neg p[w_0\mathsf{R}^-w_1\leqslant w_1\!:\!1\!:\!p[\frownie]]]
]]]]]]
\end{forest}}
\end{minipage}
\hfill
\begin{minipage}{.45\linewidth}
\[\xymatrix{w_0\ar@/^1pc/[rr]|{R^+=1}\ar@/_1pc/[rr]|{R^-=\frac{1}{2}}&&w_1:p=\left(\frac{1}{2},0\right)}\]
\end{minipage}
\end{center}
We can now extract a model from the complete open branch marked with $\frownie$ s.t.\ $v_2(\neg\blacksquare p\!\rightarrow\!\blacksquare\neg p,w_0)>0$. We use $w$'s that occur thereon as the carrier and assign the values of variables and relations so that they correspond to $\lesssim$.

\begin{theorem}[$\TfbinfoGsquare$ completeness]\label{theorem:TinfoG2completeness}
$\phi$ is \emph{strongly} valid in $\infoGsquare$ iff there is a tableau proof of $\phi$.
\end{theorem}
\begin{proof}
The proof is an easy adaptation of~\cite[Theorem~3]{BilkovaFrittellaKozhemiachenko2022IJCAR}, whence we provide only a sketch thereof. The skipped steps can be seen in Section~\ref{subsec:completenessproof}.

To prove soundness, we need to show that if the premise of the rule is realised, then so is at least one of its conclusions. This can be done by a routine check of the rules. Note that since we work with finitely branching frames, infima and suprema from Definition~\ref{def:semantics} become maxima and minima. Since closed branches are not realisable, the result follows.

To prove completeness, we show that every complete open branch $\mathcal{B}$ is realisable. We show how to construct a realising model from the branch. First, we set $W=\{w:w\text{ occurs in }\mathcal{B}\}$. Denote the set of atomic structures appearing on $\mathcal{B}$ with $\AStr(\mathcal{B})$ and let $\mathcal{B}^+$ be the transitive closure of $\mathcal{B}$ under $\lesssim$. Now, we assign values to them. For $i\in\{1,2\}$, if $w\!:\!i\!:\!p\geqslant1\in\mathcal{B}$, we set $v_i(p,w)=1$. If $w\!:\!i\!:\!p\leqslant0\in\mathcal{B}$, we set $v_i(p,w)=0$. If $w\mathsf{S}w'<\mathfrak{X}\notin\mathcal{B}^+$, we set $w\mathsf{S}w'=1$. If $w\!:\!i\!:\!p$ or $w\mathsf{S}w'$ does not occur on $\mathcal{B}$, we set $v_i(p,w)=0$ and $w\mathsf{S}w'=0$.

For each $\str\in\AStr$, we now set
\[[\str]\!=\!\left\{\str'\left| \; \begin{matrix}\str\leqslant\str'\in\mathcal{B}^+\text{ and }\str<\str
\notin\mathcal{B}^+\\
\text{or}\\
\str\geqslant\str'\in\mathcal{B}^+\text{ and }\str>\str'\notin\mathcal{B}^+
\end{matrix}\right.\right\}\]
Denote the number of $[\str]$'s with $\#^\mathsf{str}$. Since the only possible loop in $\mathcal{B}^+$ is $\str\leqslant\str'\leqslant\ldots\leqslant\str$ where all elements belong to $[\str]$, it is clear that $\#^\mathsf{str}\leq2\cdot|\AStr(\mathcal{B})|\cdot|W|$. Put $[\str]\prec[\str']$ iff there are $\str_i\in[\str]$ and $\str_j\in[\str']$ s.t.\ $\str_i<\str_j\in\mathcal{B}^+$. We now set the valuation of these structures as follows:
\begin{align*}
\str=\dfrac{|\{[\str']\mid[\str']\prec[\str]\}|}{\#^\mathsf{str}}%\tag{$*$}\label{valuationchain}
\end{align*}
It is clear that constraints containing only atomic structures and constants are now satisfied. To show that all other constraints are satisfied, we prove that if at least one conclusion of the rule is satisfied, then so is the premise. Again, the proof is a slight modification of~\cite[Theorem~3]{BilkovaFrittellaKozhemiachenko2022IJCAR} and can be done by considering the cases of rules (the details are in Section~\ref{subsec:completenessproof}).
\end{proof}
\section{Complexity\label{sec:complexity}}
In this section, we use the tableaux to provide the upper bound on the size of falsifying (satisfying) models and prove that satisfiability and validity\footnote{Satisfiability and falsifiability (non-validity) are reducible to each other: $\phi$ is satisfiable iff ${\sim\sim}(\phi\coimplies\mathbf{0})$ is falsifiable; $\phi$ is falsifiable iff ${\sim\sim}(\mathbf{1}\coimplies\phi)$ is satisfiable.} of $\fbinfoGsquare$ are $\pspace$ complete.

The following statement follows immediately from Theorem~\ref{theorem:TinfoG2completeness}.
\begin{corollary}\label{cor:FMP}
Let $\phi\in\infobimodalLsquare$ be \emph{not $\fbinfoGsquare$ valid}, and let $k$ be the number of modalities in it. Then there is a model $\mathfrak{M}$ of the size $\leq k^{k+1}$ and depth $\leq k$ and $w\in\mathfrak{M}$ s.t.\ $v_1(\phi,w)\neq1$ or $v_2(\phi,w)\neq0$.
\end{corollary}
\begin{proof}
In Section~\ref{subsec:FMPproof}.
\end{proof}

We can now prove the $\pspace$ completeness result. The proof of $\pspace$ membership adapts the method from~\cite{BilkovaFrittellaKozhemiachenko2022IJCAR} and is inspired by the proof of the $\pspace$ membership of $\mathbf{K}$ from~\cite{BlackburndeRijkeVenema2010}. For the hardness part, we reduce the validity in $\mathbf{K}$ to $v_1$ and $v_2$ validities. We provide a sketch of the proof (the skipped steps are given in Section~\ref{subsec:PSPACEproof}).
\begin{theorem}\label{theorem:infoG2PSPACE}
$\fbinfoGsquare$ validity and satisfiability are $\pspace$ complete.
\end{theorem}
\begin{proof}
For the membership, observe from the proof of Theorem~\ref{theorem:TinfoG2completeness} that $\phi$ is satisfiable (falsifiable) on $\mathfrak{M}=\langle W,R^+,R^-,v_1,v_2\rangle$ iff all variables, $w\mathsf{R}^+w'$'s, and $w\mathsf{R}^-w'$'s have values from $\mathsf{V}=\left\{0,\frac{1}{\#^\str},\ldots,\frac{\#^\str-1}{\#^\str},1\right\}$ under which $\phi$ is satisfied (falsified).

Since $\#^\str$ is bounded from above, we can now replace constraints with labelled formulas and relational structures of the form $w\!:\!i\!:\!\phi\!=\!\mathsf{v}$ or $w\mathsf{S}w'\!=\!\mathsf{v}$ ($\mathsf{v}\in\mathsf{V}$) avoiding comparisons of values of formulas in different states. We close the branch if it contains $w\!:\!i\!:\!\psi\!=\!\mathsf{v}$ and $w\!:\!i\!:\!\psi\!=\!\mathsf{v}'$ for $\mathsf{v}\!\neq\!\mathsf{v}'$.

Now we replace the rules from Definition~\ref{def:TfbinfoGsquare} with new ones that work with labelled structures. Below, we give as an example the rules\footnote{For a value $\mathsf{v}>0$ of $\blacklozenge\phi$ at $w$, we add a new state that witnesses $\mathsf{v}$, and for a~state on the branch, we guess a~value smaller than $\mathsf{v}$. Other modal rules can be rewritten similarly.} that replace $\blacklozenge_i\!\!\lesssim$.
\begin{center}
\scriptsize{\begin{align*}
\dfrac{w\!:\!i\!:\!\blacklozenge\phi\!=\!\frac{r}{\#^\str}}{\left.\begin{matrix}w\mathsf{S}w'\!=\!1\\w\!:\!i\!:\!\phi\!=\!\frac{r}{\#^\str}\end{matrix}\right|\left.\begin{matrix}w\mathsf{S}w'\!=\!\frac{r}{\#^\str}\\w\!:\!i\!:\!\phi\!=\!1\end{matrix}\right|\ldots\left|\begin{matrix}w\mathsf{S}w'\!=\!\frac{r}{\#^\str}\\w\!:\!i\!:\!\phi\!=\!\frac{r}{\#^\str}\end{matrix}\right.}
&&
\dfrac{w\!:\!i\!:\!\blacklozenge\phi\!=\!\frac{r}{\#^\str};(w\mathsf{S}w'\text{ occurs on the branch})}{w'\!:\!i\!:\!\phi\!=\!\frac{r-1}{\#^\str}\mid w\mathsf{S}w'\!=\!\frac{r-1}{\#^\str}\mid\ldots\mid w'\!:\!i\!:\!\phi\!=\!0}\end{align*}}
\end{center}
Observe that once all rules are rewritten in this manner, we will not need to compare values of formulas \emph{in different states}.

We then proceed as follows: first, we apply the propositional rules, then \emph{one} modal rule requiring a~new state (e.g., $w_0\!:\!i\!:\!\blacklozenge\phi\!=\!\frac{r}{\#^\str}$), then the rules that use that state guessing the tableau branch when needed. By repeating this process, we are building \emph{the model branch by branch}. The model has the depth bounded by the length of $\phi$ and we work with modal formulas one by one, whence we need to store subformulas of $\phi$ and $w\mathsf{S}w'$'s with their values $O(|\phi|)$ times, so, we need only $O(|\phi|^2)$ space. Once the branch is constructed, we can delete the entries of the tableau and repeat the process with the next formula at $w_0$ that would introduce a new state.

For hardness, we reduce the $\mathbf{K}$ validity of $\{\mathbf{0},\wedge,\vee,\rightarrow,\Box,\lozenge\}$ formulas to $v_1$-validity and $v_2$-validity in $\fbinfoGsquare$. For the reduction to $v_1$-validity, we use the idea from~\cite[Theorem~21]{CaicedoMetcalfeRodriguezRogger2017}. Namely, given $\phi$, we denote with $\phi^\triangledown$ the formula whose every subformula is prenexed with ${\sim\sim}$ and where $\Box$ and $\lozenge$ are replaced with $\blacksquare$ and $\blacklozenge$. Since semantics for the G\"{o}del modal logic and for the positive support ($v_1$ valuations, Definition~\ref{def:semantics}) coincide, the result follows.

For the reduction to $v_2$-validity, we take $\phi$ and inductively define $\phi^\partial$:
\begin{align*}
p^\partial&=\mathbf{1}\coimplies(\mathbf{1}\coimplies p)\\
(\chi\circ\psi)^\partial&=\chi^\partial\bullet\psi^\partial\tag{$\circ,\bullet\in\{\wedge,\vee\}$, $\circ\neq\bullet$}\\
(\chi\rightarrow\psi)^\partial&=\psi^\partial\coimplies\chi^\partial\\
(\Box\chi)^\partial&=\blacksquare(\chi^\partial)\\
(\lozenge\chi)^\partial&=\blacklozenge(\chi^\partial)
\end{align*}
One can check by induction that for every \emph{crisp} finitely branching $\mathfrak{F}$ and every \emph{classical} valuation $\mathbf{v}$ thereon, it holds that $\mathfrak{F},\mathbf{v},w\vDash\phi$ iff $v_2(\mathbf{1}\coimplies\phi^\partial,w)=0$ and $\mathfrak{F},\mathbf{v},w\nvDash\phi$ iff $v_2(\mathbf{1}\coimplies\phi^\partial,w)=1$ provided that $v_2=\mathbf{v}$.

For the converse, let $\mathfrak{M}=\langle W,R^+,R^-,v_1,v_2\rangle$ be a $\fbinfoGsquare$ model. Let $\mathfrak{M}^!=\langle W,R^!,v^!\rangle$ be s.t.\ $wR^!w'$ iff $wR^-w'=1$ and $w\in v^!(p)$ iff $v_2(p,w)=1$. Again, it is easy to verify that for every $\mathfrak{M}$, $v_2(\phi^\partial,w)=1$ iff $\mathfrak{M}^!,w\vDash\phi$.

It follows that $\phi$ is $\mathbf{K}$-valid iff $\mathbf{1}\coimplies\phi^\partial$ is $v_2$-valid.
\end{proof}
\section{Conclusions and future work\label{sec:conclusion}}
We presented a modal expansion $\infoGsquare$ of $\Gsquare$ with non-normal modalities and provided it with Kripke semantics on bi-relational frames with two valuations. We established its connection with the bi-G\"{o}del modal logic $\KbiG$ presented in~\cite{BilkovaFrittellaKozhemiachenko2022IJCAR,BilkovaFrittellaKozhemiachenko2022IGPLarxiv} and obtained decidability and complexity results considering $\infoGsquare$ over finitely branching frames.

The next steps are as follows. First of all, we plan to explore the decidability of the full $\infoGsquare$ logic. We conjecture that it is also $\pspace$ complete. However, the standard way of proving $\pspace$ completeness of G\"{o}del modal logics described in~\cite{CaicedoMetcalfeRodriguezRogger2013,CaicedoMetcalfeRodriguezRogger2017} and used in~\cite{BilkovaFrittellaKozhemiachenko2022IGPLarxiv} to establish $\pspace$ completeness of $\KbiG$ may not be straightforwardly applicable here as the reduction from $\infoGsquare$ validity to $\KbiG$ validity can be hard to obtain for it follows immediately from Theorem~\ref{theorem:1relational} that $\infoGsquare$ lacks negation normal forms.

Second, it is interesting to design a complete Hilbert-style axiomatisation of $\infoGsquare$ and study its correspondence theory w.r.t.\ \emph{strong validity}. This can be non-trivial since $\blacksquare(p\rightarrow q)\rightarrow(\blacksquare p\rightarrow\blacksquare q)$ and $\blacklozenge(p\vee q)\rightarrow\blacklozenge p\vee\blacklozenge q$ are not $\infoGsquare$ valid, even though, it is easy to check that the following rules are sound.
\begin{align*}
\dfrac{\phi\rightarrow\chi}{\blacksquare\phi\rightarrow\blacksquare\chi}&&\dfrac{\phi\rightarrow\chi}{\blacklozenge\phi\rightarrow\blacklozenge\chi}
\end{align*}

The other direction of future research is to study global versions of $\blacksquare$ and $\blacklozenge$ as well as description logics based on them. Description G\"{o}del logics are well-known and studied~\cite{BobilloDelgadoGomez-RamiroStraccia2009,BobilloDelgadoGomez-RamiroStraccia2012} and allow for the representation of uncertain data that cannot be represented in the classical ontologies. Furthermore, they are the only decidable family of fuzzy description logics which contrasts them to e.g., \L{}ukasiewicz description (and global) logics which are not even axiomatisable~\cite{Vidal2021}. On the other hand, there are known description logics over $\BD$ (cf., e.g.~\cite{MaHitzlerLin2007}), and thus it makes sense to combine the two approaches.
\bibliographystyle{splncs04}
\bibliography{reference}
\newpage
\appendix
\section{Proofs\label{sec:longproofs}}
\subsection{Proof of Theorem~\ref{theorem:TinfoG2completeness}\label{subsec:completenessproof}}
We fill in the gaps in the sketch. First, we prove the soundness result. Since propositional rules are exactly the same as in $\mathcal{T}\left(\fbKGsquare\right)$~\cite{BilkovaFrittellaKozhemiachenko2022IJCAR}, we consider only the most interesting cases of modal rules. We tackle $\blacksquare_1\!\!\gtrsim$ and $\blacklozenge_2\!\!\gtrsim$ (cf.\ Definition~\ref{def:TfbinfoGsquare}) and show that in each case, if $\mathfrak{M}=\langle W,R^+,R^-,v_1,v_2\rangle$ realises the premise of the rule, it also realises one of its conclusions.

We begin with $\blacksquare_1\!\!\gtrsim$, assume w.l.o.g.\ that $\mathfrak{X}=w''\!:\!2\!:\!\psi$, and let $\mathfrak{M}$ realise $w\!:\!1\!:\!\blacksquare\phi\geqslant w''\!:\!2\!:\!\psi$. Now, since $R^+$ and $R^-$ are finitely branching, we have that $\min\limits_{w'\in W}\{w\mathsf{R}^+w'\rightarrow_\mathsf{G}v_1(\phi,w')\}\geq v_2(\psi,w)$, whence at each $w'\in W$ s.t.\ $wR^+w'>0$\footnote{Recall that if $u\mathsf{S}u'\notin\mathcal{B}$, we set $u\mathsf{S}u'=0$.}, either $v_1(\phi,w')\geq v_2(\psi,w'')$ or $w\mathsf{R}^+w'\geq v_2(\psi,w'')$. Thus, at least one conclusion of the rule is satisfied.

For $\blacklozenge_2\!\!\gtrsim$ we proceed similarly. Let $\mathfrak{M}$ realise $w\!:\!1\!:\!\blacklozenge\phi\geqslant w''\!:\!2\!:\!\psi$. Again, by the finite branching, we have that $\min\limits_{w'\in W}\{wR^+w'\wedge_\mathsf{G}v_1(\phi,w')\}$. Hence, there is some fresh $w'\in W$ s.t.\ $wR^+w',v_1(\phi,w')\geq v_2(\psi,w'')$. Thus, the conclusion of the rule is satisfied, as desired.

For completeness, we reason by contraposition. We show by induction on formulas that every complete open branch is realised. The case of atomic constraints holds by the construction of the realising model (recall the proof of Theorem~\ref{theorem:TinfoG2completeness}). We show that other constraints are satisfied. For that, we prove that if at least one conclusion of the rule is satisfied, then so is the premise. The propositional cases are straightforward and can be tackled in the same manner as in~\cite[Theorem~2]{BilkovaFrittellaKozhemiachenko2021}. We consider only the cases of $\blacklozenge_2\!\gtrsim$ and $\blacksquare_1\!\!\gtrsim$ and assume w.l.o.g.\ that $\mathfrak{X}=w''\!:\!2\!:\!\psi$.

For $\blacksquare_1\!\!\gtrsim$, assume that for every $w'$ s.t.\ $w\mathsf{R}^+w'$ is on the branch, either $w'\!:\!1\!:\!\phi\geqslant w''\!:\!2\!:\!\psi$ or $w\mathsf{R}^+w'\leqslant w'\!:\!1\!:\!\phi$ is realisable. Thus, by the inductive hypothesis, for every $w'\in R^+(w)$, it holds that $v_1(\phi,w')\geq v_2(\psi,w'')$ or $wR^+w'\leq v_1(\phi,w')$. Hence, $v_1(\blacksquare\phi,w)\geq v_2(\psi,w'')$ and $w\!:\!1\!:\!\blacksquare\phi\geqslant w''\!:\!2\!:\!\psi$ is realised.

For $\blacklozenge_2\!\gtrsim$, let $w\mathsf{R}^-w''\geqslant w''\!:\!2\!:\!\psi$ and $w'\!:\!1\!:\!\phi\geqslant w''\!:\!2\!:\!\psi$ be realised for some $w''\in R(w)$. By the induction hypothesis, we have that $w\mathsf{R}^-w'',v_2(\phi,w')\geq v_2(\psi,w'')$, whence, $v_2(\blacklozenge\phi,w)\geq v_2(\psi,w'')$ and thus, $w\!:\!2\!:\!\blacklozenge\phi\geqslant w''\!:\!2\!:\!\psi$.

Other rules can be considered similarly.
\subsection{Proof of Corollary~\ref{cor:FMP}\label{subsec:FMPproof}}
By theorem~\ref{theorem:TinfoG2completeness}, if $\phi$ is \emph{not $\fbinfoGsquare$ valid}, we can build a~falsifying model using tableaux. It is also clear from the rules in Definition~\ref{def:TfbinfoGsquare} that the depth of the constructed model is bounded from above by the maximal number of nested modalities in $\phi$. The width of the model is bounded by the maximal number of modalities on the same level of nesting.
\subsection{Proof of Theorem~\ref{theorem:infoG2PSPACE}\label{subsec:PSPACEproof}}
We provide the decision algorithm that utilises the rewritten rules. The algorithm is essentially the same as in~\cite{BilkovaFrittellaKozhemiachenko2022IJCAR}. Note also that it is possible to use the original calculus as a decision procedure, although it is not optimal.

Let us show how to build a satisfying model for $\phi$ using polynomial space. We begin with $w_0\!:\!1\!:\phi\!=\!1$ (the algorithm for $w_0\!:\!1\!:\phi\!=\!0$ is the same) and start applying propositional rules (first, those that do not require branching). If we implement a branching rule, we pick one branch and work only with it: either until the branch is closed, in which case we pick another one; until no more rules are applicable (then, the model is constructed); or until we need to apply a modal rule to proceed. At this stage, we need to store only the subformulas of $\phi$ with labels denoting their value at~$w_0$.

Now we guess a~modal formula (say, $w_0\!:\!2\!:\!\blacklozenge\chi\!=\!\frac{1}{\#^\str}$) whose decomposition requires an introduction of a~new state ($w_1$) and apply this rule. Then we apply all modal rules whose implementation requires that $w_0\mathsf{R}^-w_1$ occur on the branch (again, if those require branching, we guess only one branch) and start from the beginning with the propositional rules. If we reach a contradiction, the branch is closed. Again, the only new entries to store are subformulas of $\phi$ (now, with fewer modalities), their values at $w_1$, and a~relational term $w_0\mathsf{R}^-w_1$ with its value. Since the depth of the model is $O(|\phi|)$ and since we work with modal formulas one by one, we need to store subformulas of $\phi$ with their values $O(|\phi|)$ times, so, we need only $O(|\phi|^2)$ space.

Finally, if no rule is applicable and there is no contradiction, we mark $w_0\!:\!2\!:\!\blacklozenge\chi\!=\!\frac{1}{\#^\str}$ as ‘safe’. Now we \emph{delete all entries of the tableau below it} and pick another unmarked modal formula that requires an introduction of a new state. Dealing with these one by one allows us to construct the model branch by branch. But since the length of each branch of the model is bounded by $O(|\phi|)$ and since we delete \emph{branches of the model} once they are shown to contain no contradictions, we need only polynomial space.
\end{document}